\documentclass[12pt]{article}
\usepackage{amsmath, amsthm, amssymb}

\newcommand{\R}{\mathbb{R}}

\newcommand{\p}{\partial}

\numberwithin{equation}{section}

\newtheorem{theorem}{Theorem}
\newtheorem{lemma}{Lemma}
\newtheorem{proposition}{Proposition}
\theoremstyle{definition}
\newtheorem{definition}{Definition}
\newtheorem{remark}{Remark}

\title{
Strong instability of standing waves for 
a system of nonlinear Klein-Gordon equations
with quadratic interaction} 
\author{Masahito Ohta
\thanks{Department of Mathematics, Tokyo University of Science, 
1-3 Kagurazaka, Tokyo 162-8601, Japan. 
E-mail address: \texttt{mohta@rs.tus.ac.jp}}
}

\date{} 

\begin{document}

\maketitle

\begin{abstract} 
We consider a system of nonlinear Klein-Gordon equations 
with quadratic interaction in two and three space dimensions. 
The strong instability of standing wave solutions is studied 
for the system without assuming the mass resonance condition. 
\end{abstract}

\section{Introduction}

We consider the following system of nonlinear Klein-Gordon equations 
with quadratic interaction: 
\begin{equation} \label{nlkg}
\begin{cases}
\p_t^2 u_1-\Delta u_1+m_1^2 u_1=2 \overline u_1 u_2, \\
\p_t^2 u_2-\Delta u_2+m_2^2 u_2=u_1^2, 
\end{cases}
\end{equation}
where $u_1$ and $u_2$ are complex-valued functions of $(t,x) \in  \R \times \R^N$, 
$m_1$ and $m_2$ are positive constants, 
and $\overline z$ is the complex conjugate of $z$. 

We define the energy space 
$X=H^1(\R^N)^2 \times L^2 (\R^N)^2$ 
and the energy functional $E$ on $X$ by 
\begin{equation} \label{def-E} 
E(\vec u, \vec v)
=\frac{1}{2} \sum_{j=1}^{2} 
\left( \|v_j\|_{L^2}^2+\|\nabla u_j\|_{L^2}^2+m_j^2 \|u_j\|_{L^2}^2 \right) 
-G (\vec u)
\end{equation}
for  $(\vec u, \vec v) \in X$ with $\vec u=(u_1, u_2)$ and $\vec v=(v_1, v_2)$, 
where 
\begin{equation} \label{def-G} 
G (\vec u)
=\mathrm{Re} \int_{\R^N} u_1(x)^2 \, \overline {u_2(x)} \,dx
\end{equation}
is the nonlinear interaction term between $u_1$ and $u_2$. 
Note that $E$ is well defined on $X$ for $N\le 6$, 
and that $N=6$ is the critical dimension for $G$ at the level of $H^1$. 
In this paper, we use the notation 
\[
(u,v)_{L^2}
=\mathrm{Re} \int_{\R^N} u(x) \overline {v(x)} \,dx
\]
for complex-valued functions $u$, $v \in L^2(\R^N)$, 
and define the charge $Q$ by 
\begin{equation} \label{def-Q} 
Q(\vec u, \vec v)
=(v_1, i u_1)_{L^2}+2 (v_2, i u_2)_{L^2}
\end{equation}
for  $(\vec u, \vec v) \in X$, 
where $i=\sqrt{-1}$ is the imaginary unit. 

It is known that the Cauchy problem for \eqref{nlkg} 
is locally well-posed in the energy space $X$
for the case $N\le 5$ (see \cite{GV1, GV2, miyazaki}). 

\begin{proposition}
Let $N\le 5$. 
For every $(\vec u_0, \vec v_0) \in X$, 
there exists $T_{\max}=T_{\max} (\vec u_0, \vec v_0) \in (0, \infty]$ 
and a unique solution 
$(\vec u, \p_t \vec u) \in C([0, T_{\max}), X)$ of \eqref{nlkg} 
with initial condition $(\vec u(0), \p_t \vec u(0))=(\vec u_0, \vec v_0)$
such that either $T_{\max}=\infty$ $($global existence$)$ 
or $T_{\max}<\infty$ and 
$\displaystyle{\lim_{t\to T_{\max}} \| (\vec u(t), \p_t \vec u(t)) \|_{X}=\infty}$ 
$($finite time blowup$)$. 
Moreover, the solution $(\vec u(t), \p_t \vec u(t))$ satisfies 
the conservation laws of energy and charge: 
\[
E(\vec u(t), \p_t \vec u(t))=E(\vec u_0, \vec v_0), \quad  
Q(\vec u(t), \p_t \vec u(t))=Q(\vec u_0, \vec v_0)
\]
for all $t\in [0, T_{\max})$. 
\end{proposition} 

The purpose of this paper is to study the strong instability of 
standing wave solutions for \eqref{nlkg} of the form: 
\[ 
\vec u(t,x)
=\bigl( e^{i \omega t} \phi_{1}(x), \, 
e^{2 i \omega t} \phi_{2}(x) \bigr), 
\]
where $\omega\in \R$ is a constant satisfying 
$\omega^2<\min\{m_1^2, \, m_2^2/4\}$, 
and $(\phi_{1}, \phi_{2})\in H^1(\R^N)^2$ 
is a ground state of the stationary problem: 
\begin{equation} \label{SP}
\begin{cases}
-\Delta \phi_{1}+(m_1^2-\omega^2) \phi_{1}=2\overline \phi_{1} \phi_{2}, \\
-\Delta \phi_{2}+(m_2^2-4\omega^2) \phi_{2}=\phi_{1}^2. 
\end{cases}
\end{equation}

For the definition of ground states of \eqref{SP}, 
see \eqref{def-Gomega} below. 
The existence of ground states for \eqref{SP} is well known 
(see, e.g., \cite{BrLi, HOT}), 
and all ground states are radially symmetric up to a translation in $\R^N$ 
(see \cite{BJM} and references therein). 
In the following, we assume that 
$(\phi_{1, \omega}, \phi_{2,\omega})$ is a radially symmetric ground state of \eqref{SP}, 
and denote 
\[
\vec \phi_{\omega}
=\left( \phi_{1, \omega}, \, \phi_{2,\omega} \right), \quad 
\vec \psi_{\omega}
=\left( i \omega \phi_{1, \omega}, \, 2 i \omega \phi_{2,\omega} \right).
\]

\begin{definition}  
We say that the standing wave solution 
$( e^{i \omega t} \phi_{1,\omega},  e^{2 i \omega t} \phi_{2,\omega})$ 
of \eqref{nlkg} is \textit{strongly unstable} 
if for every $\varepsilon>0$ there exists $(\vec u_0, \vec v_0) \in X$  
such that 
$\|(\vec u_0, \vec v_0)-(\vec \phi_{\omega}, \vec \psi_{\omega}) \|_{X}<\varepsilon$ 
and $T_{\max} (\vec u_0, \vec v_0)<\infty$, i.e., 
the solution $\vec u(t)$ of \eqref{nlkg} with 
$(\vec u(0), \p_t \vec u(0))=(\vec u_0, \vec v_0)$ blows up in finite time.
\end{definition}

When $N \in \{4, 5\}$,  
Miyazaki \cite{miyazaki} proved that 
the standing wave solution 
$( e^{i \omega t} \phi_{1,\omega},  e^{2 i \omega t} \phi_{2,\omega})$ 
of \eqref{nlkg} is strongly unstable for all $\omega \in \R$ satisfying 
$\omega^2<\min\{m_1^2, \, m_2^2/4\}$. 
Note that $N=4$ is the critical dimension at the level of $L^2$ 
for the system \eqref{nlkg} with quadratic interaction. 

On the other hand, 
for the $L^2$-subcritical case $N \in \{2,3\}$, 
it is proved in \cite{miyazaki} that 
under the mass resonance condition $m_2=2m_1$, 
the standing wave solution 
$( e^{i \omega t} \phi_{1,\omega},  e^{2 i \omega t} \phi_{2,\omega})$ 
of \eqref{nlkg} is strongly unstable if 
\begin{equation} \label{SC0}
(5-N) \omega^2\le m_1^2.
\end{equation}

The condition \eqref{SC0} is optimal 
under the mass resonance condition $m_2=2m_1$. 
In fact, 
by the same argument as in Shatah \cite{shatah83} 
for the single nonlinear Klein-Gordon equation: 
\begin{equation} \label{nlkg0} 
\p_t^2 u-\Delta u+u=|u|^{p-1} u, \quad 
(t,x) \in \R \times \R^N, 
\end{equation}
one can prove that if $m_2=2m_1$ 
and $\omega$ satisfies 
$m_1^2/(5-N) <\omega^2<m_1^2$, 
then the set 
$\{ (\vec \phi, \vec \psi): \vec \phi=(\phi_1, \phi_2) \in \mathcal G_{\omega}, \, 
\vec \psi=(i \omega \phi_1, 2i \omega \phi_2)\}$ 
is stable under the flow of \eqref{nlkg}, 
where $\mathcal G_{\omega}$ is the set of ground states of \eqref{SP} 
defined by \eqref{def-Gomega}. 

The main result of this paper is the following. 

\begin{theorem} \label{thm1}
Let $N \in \{2,3\}$. 
Assume that $m_1$ and $m_2$ be positive constants, 
and $\omega\in \R$ satisfies 
\begin{equation} \label{SC1}
(5-N) \omega^2\le \min \bigl\{ m_1^2, \, m_2^2/4 \bigr\}. 
\end{equation}
Then, the standing wave solution 
$( e^{i \omega t} \phi_{1,\omega},  \, e^{2 i \omega t} \phi_{2,\omega})$ 
of \eqref{nlkg} is strongly unstable, 
where $(\phi_{1, \omega}, \, \phi_{2,\omega})$ 
is a radially symmetric ground state of \eqref{SP}. 
\end{theorem}

\begin{remark}
In Theorem \ref{thm1}, 
the assumption \eqref{SC1} can be replaced by the condition 
\begin{equation}\label{SC2} 
\bigl\{ m_1^2-(5-N) \omega^2 \bigr\} \|\phi_{1,\omega}\|_{L^2}^2
+\bigl\{ m_2^2-4(5-N) \omega^2 \bigr\} \|\phi_{2,\omega}\|_{L^2}^2
\ge 0
\end{equation}
(see Sections \ref{sect4} and \ref{sect5}). 
However, it seems difficult to confirm \eqref{SC2} 
other than the case \eqref{SC1}. 
\end{remark}

\begin{remark}
The condition \eqref{SC2} is equivalent to 
\begin{equation} \label{SC11}
\p_{\lambda}^2 
E \bigl( (\vec \phi_{\omega})^{\lambda}, \, (\vec \psi_{\omega})_{\lambda} \bigr) 
|_{\lambda=1}\le 0 
\end{equation}
(see Lemma \ref{blowup2} in Section \ref{sect4}), 
where $u^{\lambda}$ and $v_{\lambda}$ are scaling defined by 
\begin{equation} \label{scale1}
u^{\lambda}(x)=\lambda^2 u(\lambda x), \quad 
v_{\lambda}(x)=\lambda^{N-2} v(\lambda x)
\end{equation} 
for $\lambda>0$ and $x\in \R^N$. 
The scaling \eqref{scale1} leaves the charge $Q$ invariant, 
and plays a central role in this paper. 
\end{remark}

\begin{remark}
Although the condition \eqref{SC1} coincides with \eqref{SC0} 
under the mass resonance condition $m_2=2m_1$, 
it is not clear whether \eqref{SC1} or \eqref{SC2} is optimal 
in the case $m_2\ne 2m_1$. 
It will be an interesting and difficult problem to investigate the stability of 
$( e^{i \omega t} \phi_{1,\omega},  \, e^{2 i \omega t} \phi_{2,\omega})$ 
for \eqref{nlkg} 
in the case where $m_2\ne 2m_1$ and 
$\max \bigl\{ m_1^2, \, m_2^2/4\bigr\}<(5-N)\omega^2$. 
\end{remark}

The proof in Miyazaki \cite{miyazaki} for the system \eqref{nlkg} 
is based on the argument in Ohta and Todorova \cite{OT} 
for the single nonlinear Klein-Gordon equation \eqref{nlkg0}, 
and it seems difficult to apply the argument in \cite{miyazaki, OT} 
to the case $m_2\ne 2m_1$ 
(see Remark \ref{remark5} in Section \ref{sect:CGS}). 
The novel idea for proving Theorem \ref{thm1} is found 
in the key Proposition \ref{key}, 
where we use a similar idea introduced by the author \cite{ohta18} 
for the nonlinear Schr\"odinger equation with harmonic potential: 
\begin{equation} \label{nls-harmonic}
i\p_t u=-\Delta u+|x|^2 u-|u|^{p-1}u, \quad 
(t,x) \in \R \times \R^N. 
\end{equation}
It is interesting that similar structures appear in 
the nonlinear Schr\"odinger equation \eqref{nls-harmonic} 
with harmonic potential 
and the nonlinear Klein-Gordon equations \eqref{nlkg} and \eqref{nlkg0} 
(see Section \ref{sect5}). 

The rest of the paper is organized as follows. 
In Section \ref{sect:virial}, we recall the virial identity for \eqref{nlkg} 
and its localized version. 
It is important to note that 
the virial identity \eqref{virial1} is closely related to the scaling \eqref{scale1} 
which leaves the charge $Q$ invariant. 
This natural point of view has been overlooked in previous papers \cite{miyazaki, OT}. 
In Section \ref{sect:CGS}, 
we give some variational characterizations of ground states for \eqref{SP}. 
Then, the main Theorem \ref{thm1} is proved in Section \ref{sect4} 
using the key Proposition \ref{key}. 
Finally, we give the proof of Proposition \ref{key} in Section \ref{sect5}. 

\section{Virial identity} \label{sect:virial}

Throughout this section, we assume that $N\in \{2, 3\}$. 
The virial identity is closely related to the scaling \eqref{scale1}: 
$u^{\lambda}(x)=\lambda^2 u(\lambda x)$, 
$v_{\lambda}(x)=\lambda^{N-2} v(\lambda x)$. 

For $\vec u=(u_1, u_2)$, $\vec v=(v_1, v_2)$ and $\lambda>0$, 
we define 
\[
\vec u^{\lambda}=\left( (u_1)^{\lambda}, \, (u_2)^{\lambda} \right), \quad 
\vec v_{\lambda}=\left( (v_1)_{\lambda}, \, (v_2)_{\lambda} \right).
\]
Then, we have 
\begin{align}
&Q(\vec u^{\lambda}, \vec v_{\lambda})=Q(\vec u, \vec v), 
\label{Qlambda} \\
&E(\vec u^{\lambda}, \vec v_{\lambda})
=\lambda^{-\alpha} K (\vec v)
+\lambda^{\alpha} M (\vec u) 
-\lambda^{\alpha+2} L(\vec u), 
\label{Elambda} 
\end{align}
where $\alpha=4-N$, and 
\begin{align}
&K (\vec v)
=\frac{1}{2} \sum_{j=1}^{2} \|v_j\|_{L^2}^2, \quad 
M (\vec u)
=\frac{1}{2} \sum_{j=1}^{2} m_j^2 \|u_j\|_{L^2}^2, 
\nonumber \\
&L (\vec u)
=-\frac{1}{2} \sum_{j=1}^{2} \|\nabla u_j\|_{L^2}^2
+G(\vec u), 
\label{def-L}  
\end{align}
and $G$ is defined by \eqref{def-G}. 
Moreover, we define
\begin{equation} \label{def-H}
H (\vec u, \vec v)
=-\alpha K (\vec v) +\alpha M (\vec u) -(\alpha+2) L (\vec u). 
\end{equation} 
Note that 
$H (\vec u, \vec v)
=\p_{\lambda} E(\vec u^{\lambda}, \vec v_{\lambda}) |_{\lambda=1}$. 
Then, a solution $\vec u(t)$ of \eqref{nlkg} formally satisfies 
the virial identity or the broken dilation identity: 
\begin{align} 
&-\frac{d}{dt} \sum_{j=1}^{2} \mathrm{Re} \int_{\R^N} 
\bigl\{ x \cdot \nabla u_j(t,x)+2 u_j(t,x) \bigr\} \p_t \overline{u_j(t,x)} \, dx 
\nonumber \\
&=H(\vec u(t), \p_t \vec u(t))
\label{virial1} 
\end{align}
(see \cite{miyazaki} and also \cite{OT, SS}), 
where note that 
$\p_{\lambda} u^{\lambda}(x) |_{\lambda=1}=x \cdot \nabla u(x)+2 u(x)$. 

Since the integral on the left-hand side of \eqref{virial1} is not well defined 
on the energy space $X$, 
we need to approximate the weight function $x$ in \eqref{virial1} 
by suitable bounded functions. 

\begin{definition}
Let $\Phi \in C^1 [0, \infty)$ be a function satisfying 
\[
\Phi (r)=
\begin{cases}
N & \mbox{for} \hspace{3mm} 0\le r \le 1, \\
0 & \mbox{for} \hspace{3mm} r\ge 2, 
\end{cases}
\quad 
\Phi'(r)\le 0 
\hspace{3mm} \mbox{for} \hspace{3mm} 
1\le r \le 2. 
\]

For $\rho>0$, we define 
\[
\Phi_{\rho} (r)=\Phi \left( \frac{r}{\rho} \right), \quad 
\Psi_{\rho} (r)=\frac{1}{r^{N-1}} \int_{0}^{r} \Phi_{\rho} (s) s^{N-1} \,ds, 
\]
and 
\begin{align*}
I_{\rho} (\vec u, \vec v)
&=\sum_{j=1}^{2} \mathrm{Re} \int_{\R^N} 
\frac{\Psi_{\rho} (|x|)}{|x|} 
x \cdot \nabla u_j (t,x) \p_t \overline{u_j (t,x)} \,dx \\
&\hspace{5mm}
+\sum_{j=1}^{2} \mathrm{Re} \int_{\R^N} 
\frac{1}{2} \bigl\{ \Phi_{\rho} (|x|)+4-N \bigr\} 
u_j(t,x) \p_t \overline{u_j(t,x)}\,dx 
\end{align*}
for $(\vec u, \vec v) \in X$. 
\end{definition}

\begin{remark}
The functions $\Phi_{\rho}$ and $\Psi_{\rho}$ satisfy the following properties: 
\begin{align*}
&\Phi_{\rho}(r)=N, \quad \Psi_{\rho}(r)=r
\hspace{3mm} \mbox{for} \hspace{3mm} 
0\le r\le \rho, \\
&0\le \Phi_{\rho} (r)\le N, \quad 
0\le \Psi_{\rho}(r)\le 2 \rho 
\hspace{3mm} \mbox{for} \hspace{3mm} 
r\ge 0 
\end{align*}
(see \cite[Lemma 7]{OT}). 
\end{remark}

The following proposition is proved by Miyazaki \cite{miyazaki}. 

\begin{proposition} \label{virial2} 
Let $(\vec u, \p_t \vec u) \in C([0, T_{\max}), X)$ 
be a radially symmetric solution of \eqref{nlkg}. 
Then, there exist positive constants $C_1$ and $C_2$, 
which depend only on $N$, such that 
\begin{align}
&-\frac{d}{dt} I_{\rho} (\vec u(t), \p_t \vec u(t)) 
\label{virial3} \\
&\le H (\vec u(t), \p_t \vec u(t)) 
+C_1 \rho^{-1} \|\vec u(t)\|_{H^1}^2
+C_2 \rho^{-(N-1)/2} \|\vec u(t)\|_{H^1}^3
\nonumber 
\end{align}
for all $t\in [0, T_{\max})$ and $\rho>0$. 
\end{proposition}

For the proof of Proposition \ref{virial2}, 
see Lemma 2.2 of \cite{miyazaki} and also Lemma 8 of \cite{OT}. 
Note that the uniform decay estimate of Strauss \cite{strauss} 
for radially symmetric functions in $H^1(\R^N)$: 
\[
\sup_{x\in \R^N} |x|^{(N-1)/2} |u(x)|\le C \|u\|_{H^1}
\]
is used in the proof of Proposition \ref{virial2}. 
When $N=1$, the function $\rho^{-(N-1)/2}$ 
on the right-hand side of \eqref{virial3} 
does not decay to $0$ as $\rho \to \infty$. 
Therefore, we cannot treat the case $N=1$ in this paper 
as well as in \cite{miyazaki, OT}. 

\section{Characterizations of ground states} \label{sect:CGS} 

In this section, we study variational characterizations 
of ground states of \eqref{SP}. 
Throughout this section, 
we assume that $N\le 3$ and $\omega\in \R$ satisfies 
$\omega^2<\min\{m_1^2, \, m_2^2/4\}$. 

For $\vec u=(u_1, u_2) \in H^1(\R^N)^2$, we define 
\begin{align} 
&J_{\omega} (\vec u) 
=M_{\omega} (\vec u)-L(\vec u), 
\nonumber \\
&M_{\omega} (\vec u) 
=\frac{m_1^2-\omega^2}{2} \|u_1\|_{L^2}^2
+\frac{m_2^2-4\omega^2}{2} \|u_2\|_{L^2}^2, 
\label{Momega} 
\end{align}
where $L$ is defined by \eqref{def-L}. 
Then, 
$\vec \phi \in H^1(\R^N)^2$ is a solution of \eqref{SP} 
if and only if $J_{\omega}' ( \vec \phi )=0$. 
Let 
\[
\mathcal N_{\omega}
=\{ \vec u\in H^1(\R^N)^2 \setminus \{\vec 0\}: 
J_{\omega}' ( \vec u)=0 \}
\]
be the set of non-trivial solutions of \eqref{SP}, 
and we define the set of ground states of \eqref{SP} by 
\begin{equation} \label{def-Gomega}
\mathcal G_{\omega}
=\{ \vec \phi \in \mathcal N_{\omega} : 
J_{\omega} (\vec \phi) \le J_{\omega} (\vec u) 
\hspace{2mm} \mbox{for all} \hspace{2mm} 
\vec u\in \mathcal N_{\omega} \}. 
\end{equation} 

We define the Nehari functional $K_{\omega}$ by 
\[
K_{\omega}(\vec u)
=\p_{\lambda} J_{\omega} (\lambda \vec u) |_{\lambda=1} 
=2 M_{\omega}(\vec u)
+\sum_{j=1}^{2} \|\nabla u_j\|_{L^2}^2-3 G(\vec u). 
\]
By the standard variational method, we see that 
the set $\mathcal G_{\omega}$ is not empty, 
and $\vec \phi_{\omega} \in \mathcal G_{\omega}$ satisfies 
\begin{equation} \label{VC1} 
J_{\omega} (\vec \phi_{\omega})
=\inf \{J_{\omega} (\vec v) : 
\vec v \in H^1(\R^N)^2 \setminus \{ \vec 0 \}, \hspace{2mm} 
K_{\omega} (\vec v)=0 \}. 
\end{equation} 

Next, we define another functional $P_{\omega}$ by 
\begin{equation} \label{Pomega}
P_{\omega}(\vec u)
=\p_{\lambda} J_{\omega} \bigl( \vec u^{\lambda} \bigr) |_{\lambda=1} 
=\alpha M_{\omega} (\vec u) -(\alpha+2) L(\vec u). 
\end{equation}
Recall that $\alpha=4-N$ and 
$J_{\omega} \bigl( \vec u^{\lambda} \bigr)
=\lambda^{\alpha} M_{\omega} (\vec u)-\lambda^{\alpha+2} L(\vec u)$, 
and note that 
$P_{\omega} (\vec \phi_{\omega})
=\p_{\lambda} J_{\omega} 
\bigl( (\vec \phi_{\omega})^{\lambda} \bigr) |_{\lambda=1} 
=0$. 

The following characterization of ground states 
in terms of $P_{\omega}$ plays an important role 
in the proof of Theorem \ref{thm1}. 

\begin{lemma} \label{VC2} 
Let $\vec \phi_{\omega} \in \mathcal G_{\omega}$. 
If $\vec u \in H^1(\R^N)^2 \setminus \{ \vec 0 \}$ satisfies 
$P_{\omega} (\vec u)=0$, 
then 
$J_{\omega} (\vec \phi_{\omega}) \le J_{\omega} (\vec u)$. 
\end{lemma}

\begin{proof}
Although Lemma \ref{VC2} has been proved by Miyazaki \cite[Section 4]{miyazaki}, 
we give a simple proof based on the idea of Jeanjean and Le Coz \cite{JJLC}. 

Let $\vec u \in H^1(\R^N)^2 \setminus \{ \vec 0 \}$ satisfy 
$P_{\omega} (\vec u)=0$. 
Then, since 
\[
\p_{\lambda} J_{\omega} \bigl( \vec u^{\lambda} \bigr) |_{\lambda=1}
=P_{\omega}(\vec u)=0, \quad 
L (\vec u)=\frac{\alpha}{\alpha+2} M_{\omega} (\vec u)>0, 
\]
we see that the function 
$\lambda \mapsto J_{\omega} \bigl( \vec u^{\lambda} \bigr) 
=\lambda^{\alpha} M_{\omega} (\vec u)-\lambda^{\alpha+2} L(\vec u)$ 
is increasing on the interval $(0,1)$ and decreasing on $(1, \infty)$. 

Thus, we have 
$J_{\omega} \bigl( \vec u^{\lambda} \bigr) \le J_{\omega} (\vec u)$ 
for all $\lambda \in (0, \infty)$. 

Moreover, since 
\[
K_{\omega} \bigl( \vec u^{\lambda} \bigr) 
=2\lambda^{\alpha} M_{\omega}(\vec u)
+\lambda^{\alpha+2}
\left( \sum_{j=1}^{2} \|\nabla u_j\|_{L^2}^2-3 G(\vec u) \right)
\]
and 
\[
\sum_{j=1}^{2} \|\nabla u_j\|_{L^2}^2-3 G(\vec u) 
=-3 L(\vec u)-\frac{1}{2}\sum_{j=1}^{2} \|\nabla u_j\|_{L^2}^2<0, 
\]
there exists $\lambda_0 \in (0, \infty)$ such that 
$K_{\omega} \bigl( \vec u^{\lambda_0} \bigr)=0$. 

Hence, it follows from \eqref{VC1} that 
$J_{\omega} (\vec \phi_{\omega})
\le J_{\omega} \bigl( \vec u^{\lambda_0} \bigr)
\le J_{\omega} (\vec u)$. 
\end{proof}

\begin{remark}\label{remark5} 
Since $P_{\omega} (\vec \phi_{\omega})=0$ 
and 
$(\alpha+2) J_{\omega} (\vec v)-P_{\omega} (\vec v)=2 M_{\omega} (\vec v)$ 
for $\vec v \in  H^1(\R^N)^2$, 
it follows from Lemma \ref{VC2} that 
if $\vec u \in H^1(\R^N)^2 \setminus \{ \vec 0 \}$ satisfies 
$P_{\omega} (\vec u)=0$, 
then $M_{\omega} (\vec \phi_{\omega})\le M_{\omega} (\vec u)$. 
By the definition \eqref{Momega} of $M_{\omega}$, 
under the mass resonance condition $m_2=2 m_1$, 
the inequality $M_{\omega} (\vec \phi_{\omega})\le M_{\omega} (\vec u)$ 
is equivalent to 
$\| \phi_{1, \omega} \|_{L^2}^2+4 \| \phi_{2, \omega} \|_{L^2}^2
\le \| u_1 \|_{L^2}^2+4 \| u_2 \|_{L^2}^2$. 
This estimate plays an important role in the case $m_2=2 m_1$, 
but does not work well in the case $m_2 \ne 2 m_1$ 
(see Section 6.1 of \cite{miyazaki}). 
\end{remark}

In this paper, 
in addition to the estimate in terms of $M_{\omega}$, 
we use the estimate in terms of $L$, 
as in the following lemma. 

\begin{lemma} \label{VC3} 
Let $\vec \phi_{\omega} \in \mathcal G_{\omega}$. 
If $\vec u \in H^1(\R^N)^2$ satisfies $P_{\omega} (\vec u)<0$, 
then 
\[
M_{\omega} (\vec \phi_{\omega})< M_{\omega} (\vec u), \quad 
L (\vec \phi_{\omega})< L(\vec u).
\]
\end{lemma}

\begin{proof}
Since 
$P_{\omega} \bigl( \vec u^{\lambda} \bigr)
=\alpha \lambda^{\alpha} M_{\omega} (\vec u) 
-(\alpha+2) \lambda^{\alpha+2} L(\vec u)$ 
and $P_{\omega} (\vec u)<0$, 
we see that $L(\vec u)>0$ and 
there exists $\lambda_0 \in (0,1)$ such that 
$P_{\omega} \bigl( \vec u^{\lambda_0} \bigr)=0$. 

Moreover, since 
$P_{\omega} (\vec \phi_{\omega})=0$ and 
\[
(\alpha+2) J_{\omega} (\vec v)-P_{\omega} (\vec v)=2 M_{\omega} (\vec v), \quad 
\alpha J_{\omega} (\vec v)-P_{\omega} (\vec v)=2 L(\vec v) 
\]
for $\vec v \in H^1(\R^N)^2$, 
it follows from Lemma \ref{VC2} that 
\begin{align*}
&2 M_{\omega} (\vec \phi_{\omega})
=(\alpha+2) J_{\omega} (\vec \phi_{\omega}) \\
&\hspace{16mm} 
\le (\alpha+2) J_{\omega} \bigl( \vec u^{\lambda_0} \bigr) 
=2 M_{\omega} \bigl( \vec u^{\lambda_0} \bigr) 
=2 \lambda_0^{\alpha} M_{\omega} (\vec u)
<2 M_{\omega} (\vec u), \\
&2 L (\vec \phi_{\omega})
=\alpha J_{\omega} (\vec \phi_{\omega})
\le \alpha J_{\omega} \bigl( \vec u^{\lambda_0} \bigr)
=2 L \bigl( \vec u^{\lambda_0} \bigr) 
=2 \lambda_0^{\alpha+2} L (\vec u)
<2 L (\vec u).
\end{align*}

This completes the proof. 
\end{proof} 

\section{Proof of Theorem \ref{thm1}} \label{sect4}

In this section, we give the proof of Theorem \ref{thm1}. 
Throughout this section, we assume that $N\in \{2,3\}$, 
$\omega^2<\min\{m_1^2, \, m_2^2/4\}$, 
and $\vec \phi_{\omega}=(\phi_{1, \omega}, \phi_{2,\omega}) 
\in \mathcal G_{\omega}$ is radially symmetric. 
We put $\alpha=4-N$ and 
$\vec \psi_{\omega}
=\left( i \omega \phi_{1, \omega}, \, 2 i \omega \phi_{2,\omega} \right)$. 

First, we prove that 
the conditions \eqref{SC2} and \eqref{SC11} are equivalent to each other. 
The condition (iii) in Lemma \ref{blowup2} 
is used in the proof of Proposition \ref{key} 
(see \eqref{flam0}). 

\begin{lemma} \label{blowup2} 
The following conditions are equivalent to each other. 
\begin{itemize}
\item [$(\mathrm{i})$] \hspace{0mm} 
$\p_{\lambda}^2 
E \bigl( (\vec \phi_{\omega})^{\lambda}, \, (\vec \psi_{\omega})_{\lambda} \bigr) 
|_{\lambda=1}\le 0$. 
\item [$(\mathrm{ii})$] \hspace{0mm}
$\bigl\{ m_1^2-(5-N) \omega^2 \bigr\} \|\phi_{1,\omega}\|_{L^2}^2
+\bigl\{ m_2^2-4(5-N) \omega^2 \bigr\} \|\phi_{2,\omega}\|_{L^2}^2
\ge 0$. 
\item [$(\mathrm{iii})$] \hspace{0mm}
$\alpha^2 K (\vec \psi_{\omega})\le (\alpha+2) L ( \vec \phi_{\omega})$. 
\end{itemize}
\end{lemma}

\begin{proof}
First, by \eqref{Elambda}, we have 
\begin{align}
&E \bigl( (\vec \phi_{\omega})^{\lambda}, 
\, (\vec \psi_{\omega})_{\lambda} \bigr) 
=\lambda^{-\alpha} K (\vec \psi_{\omega})
+\lambda^{\alpha} M (\vec \phi_{\omega})
-\lambda^{\alpha+2} L(\vec \phi_{\omega}), 
\label{El1} \\
&\p_{\lambda}^2 E \bigl( (\vec \phi_{\omega})^{\lambda}, 
\, (\vec \psi_{\omega})_{\lambda} \bigr) 
|_{\lambda=1} 
\label{El0} \\
&=\alpha (\alpha+1) K (\vec \psi_{\omega})
+\alpha (\alpha-1) M (\vec \phi_{\omega})
-(\alpha+1) (\alpha+2) L(\vec \phi_{\omega}). 
\nonumber 
\end{align}
Next, by \eqref{Momega} and \eqref{Pomega}, 
we have 
\begin{align}
&0=P_{\omega} (\vec \phi_{\omega})
=\alpha M_{\omega} (\vec \phi_{\omega})
-(\alpha+2) L (\vec \phi_{\omega}), 
\label{El2} \\
&K (\vec \psi_{\omega})
=\frac{\omega^2}{2} \| \phi_{1,\omega} \|_{L^2}^2
+\frac{4\omega^2}{2} \| \phi_{2,\omega} \|_{L^2}^2
=M (\vec \phi_{\omega})-M_{\omega} (\vec \phi_{\omega}), 
\label{El3}
\end{align}
which imply that 
\begin{align*}
&-\frac{1}{\alpha} 
\p_{\lambda}^2 E \bigl( (\vec \phi_{\omega})^{\lambda}, 
\, (\vec \psi_{\omega})_{\lambda} \bigr) 
|_{\lambda=1} \\
&=-(\alpha+1) K (\vec \psi_{\omega})
-(\alpha-1) M (\vec \phi_{\omega})
+(\alpha+1) M_{\omega} (\vec \phi_{\omega}) \\
&=-2 \alpha M(\vec \phi_{\omega})
+2 (\alpha+1) M_{\omega}(\vec \phi_{\omega}) \\
&=\bigl\{ m_1^2-(\alpha+1) \omega^2 \bigr\} \|\phi_{1,\omega}\|_{L^2}^2
+\bigl\{ m_2^2-4(\alpha+1) \omega^2 \bigr\} \|\phi_{2,\omega}\|_{L^2}^2.
\end{align*}
Since $\alpha=4-N$, we see that (i) is equivalent to (ii). 

Furthermore, by \eqref{El2} and \eqref{El3}, 
we have 
\begin{equation} \label{El4}
H (\vec \phi_{\omega}, \vec \psi_{\omega}) 
=-\alpha K (\vec \psi_{\omega})
+\alpha M (\vec \phi_{\omega})
-(\alpha+2) L(\vec \phi_{\omega}) 
=P_{\omega} (\vec \phi_{\omega})=0, 
\end{equation} 
and 
\begin{align*}
\p_{\lambda}^2 E \bigl( (\vec \phi_{\omega})^{\lambda}, 
\, (\vec \psi_{\omega})_{\lambda} \bigr) 
|_{\lambda=1} 
&=\p_{\lambda}^2 E \bigl( (\vec \phi_{\omega})^{\lambda}, 
\, (\vec \psi_{\omega})_{\lambda} \bigr) 
|_{\lambda=1} 
-(\alpha-1) H(\vec \phi_{\omega}, \vec \psi_{\omega}) \\
&=2 \alpha^2 K (\vec \psi_{\omega})
-2 (\alpha+2) L(\vec \phi_{\omega}). 
\end{align*}
Thus, we see that (i) is equivalent to (iii). 
\end{proof}

Next, we define the action $S_{\omega}$ by 
\begin{equation} \label{Somega}
S_{\omega} (\vec u, \vec v)
=E(\vec u, \vec v)-\omega Q(\vec u, \vec v). 
\end{equation} 
Since $E$ and $Q$ are conserved quantities for \eqref{nlkg}, 
so is $S_{\omega}$. 
Moreover, we have 
\begin{equation} \label{Somega0}
S_{\omega} (\vec u, \vec v)
=J_{\omega} (\vec u)
+\frac{1}{2} \| v_1-i \omega u_1 \|_{L^2}^2
+\frac{1}{2} \| v_2-2 i \omega u_2 \|_{L^2}^2. 
\end{equation} 

We define the set 
\[
\mathcal A_{\omega} 
=\{ (\vec u, \vec v) \in X: 
S_{\omega} (\vec u, \vec v)< S_{\omega} (\vec \phi_{\omega}, \vec \psi_{\omega}),  
\hspace{1mm} 
P_{\omega} (\vec u)<0 \}. 
\]

\begin{lemma} \label{InvSet1} 
Let $(\vec u_0, \vec v_0) \in \mathcal A_{\omega}$, 
and let $(\vec u(t), \p_t \vec u(t))$ be the solution of \eqref{nlkg} 
with $(\vec u(0), \p_t \vec u(0))=(\vec u_0, \vec v_0)$. 
Then, $P_{\omega} (\vec u(t))<0$ 
for all $t\in [0, T_{\max})$. 
\end{lemma}

\begin{proof}
Suppose that the conclusion of Lemma \ref{InvSet1} does not hold. 
Then, there exists $t_0 \in (0, T_{\max})$ such that 
$P_{\omega} (\vec u(t_0))=0$ 
and $P_{\omega} (\vec u(t))<0$ for $t \in [0, t_0)$. 

By Lemma \ref{VC3} 
and the continuity of the function $t\mapsto M_{\omega} (\vec u(t))$, 
we see that 
$0<M_{\omega} (\vec \phi_{\omega})\le M_{\omega} (\vec u(t_0))$. 
Thus, we have $\vec u(t_0)\ne \vec 0$, 
and it follows from Lemma \ref{VC2} that 
$J_{\omega}(\vec \phi_{\omega}) \le J_{\omega}(\vec u(t_0))$. 

Moreover, by \eqref{Somega0} 
and the conservation law of $S_{\omega}$, we have 
\begin{equation} \label{Somega1}
S_{\omega}(\vec \phi_{\omega}, \vec \psi_{\omega})
=J_{\omega}(\vec \phi_{\omega}) 
\le J_{\omega}(\vec u(t_0)) 
\le S_{\omega}(\vec u(t_0)), \p_t \vec u(t_0)) 
=S_{\omega} (\vec u_0, \vec v_0), 
\end{equation}
which contradicts the assumption that $(\vec u_0, \vec v_0) \in \mathcal A_{\omega}$. 

This completes the proof. 
\end{proof} 

The following proposition is the key to the proof of Theorem \ref{thm1}. 

\begin{proposition} \label{key} 
Assume \eqref{SC2}. 
If $(\vec u, \vec v) \in X$ satisfies 
\begin{equation} \label{HQP}
H (\vec u, \vec v) \le 0, \quad 
Q (\vec u, \vec v) = Q (\vec \phi_{\omega}, \vec \psi_{\omega}), \quad 
P_{\omega} (\vec u)<0,  
\end{equation} 
then 
\begin{equation} \label{HQP2}
\alpha E (\vec \phi_{\omega}, \vec \psi_{\omega})
\le \alpha E (\vec u, \vec v) - H (\vec u, \vec v).
\end{equation} 
\end{proposition} 

We give the proof of Proposition \ref{key} in the next section. 
Once we have the key Proposition \ref{key}, 
we can prove Theorem \ref{thm1} in the same way as in \cite{miyazaki, OT}. 
For the sake of completeness, 
we give the proof of Theorem \ref{thm1} using Proposition \ref{key}. 
We define the set 
\begin{align*}
\mathcal B_{\omega} 
=\{ (\vec u, \vec v) \in X: 
&\hspace{1mm} 
E(\vec u, \vec v)< E (\vec \phi_{\omega}, \vec \psi_{\omega}),  \hspace{1mm}
Q(\vec u, \vec v)=Q (\vec \phi_{\omega}, \vec \psi_{\omega}),  \\
&\hspace{5mm} 
H(\vec u, \vec v)<0, \hspace{1mm} 
P_{\omega} (\vec u)<0 \}. 
\end{align*}

\begin{lemma} \label{InvSet2} 
Assume \eqref{SC2}. 
Let $(\vec u_0, \vec v_0) \in \mathcal B_{\omega}$, 
and let $(\vec u(t), \p_t \vec u(t))$ be the solution of \eqref{nlkg} 
with $(\vec u(0), \p_t \vec u(0))=(\vec u_0, \vec v_0)$. 
Then, 
\[P_{\omega} (\vec u(t))<0, \quad 
H (\vec u(t), \p_t \vec u(t))<0
\]
for all $t\in [0, T_{\max})$. 
\end{lemma}

\begin{proof}
First, since 
$(\vec u_0, \vec v_0) \in \mathcal B_{\omega} \subset \mathcal A_{\omega}$, 
it follows from Lemma \ref{InvSet1} that $P_{\omega} (\vec u(t))<0$ 
for all $t \in [0, T_{\max})$. 

Next, 
suppose that there exists $t_0 \in (0, T_{\max})$ such that 
$H (\vec u(t_0), \p_t \vec u(t_0))=0$. 
Then, since $P_{\omega} (\vec u(t_0))<0$, 
it follows from Proposition \ref{key} 
and the conservation laws of $E$ and $Q$ that 
\[
E (\vec \phi_{\omega}, \vec \psi_{\omega})
\le E (\vec u(t_0), \p_t \vec u(t_0))-\frac{1}{\alpha} H (\vec u(t_0), \p_t \vec u(t_0))
=E (\vec u_0, \vec v_0),
\]
which contradicts the assumption that 
$(\vec u_0, \vec v_0) \in \mathcal B_{\omega}$. 

Hence, we conclude that $H (\vec u(t), \p_t \vec u(t))<0$ 
for all $t \in [0, T_{\max})$. 
\end{proof} 

\begin{lemma} \label{blowup1} 
Assume \eqref{SC2}. 
Assume that  $\vec u_0$ and $\vec v_0$ are radially symmetric, 
and $(\vec u_0, \vec v_0) \in \mathcal B_{\omega}$. 
Then, $T_{\max} (\vec u_0, \vec v_0)<\infty$, i.e., 
the solution $\vec u(t)$ of \eqref{nlkg} 
with $(\vec u(0), \p_t \vec u(0))=(\vec u_0, \vec v_0)$ 
blows up in finite time. 
\end{lemma}

\begin{proof}
Suppose that $T_{\max} (\vec u_0, \vec v_0)=\infty$. 
Then, it follows from Lemma \ref{InvSet2} that 
$P_{\omega} (\vec u(t))<0$ for all $t \in [0, \infty)$. 
Moreover, by Lemma \ref{VC3} and \eqref{def-L}, 
we see that 
\[
0<L (\vec \phi_{\omega})< L(\vec u(t)) \le G (\vec u(t)) 
\]
for all $t \in [0, \infty)$. 
Thus, by Proposition 5 of \cite{miyazaki} 
(see also \cite{cazenave85}), 
we have 
\[
C_0:=\sup_{t \in [0, \infty)} \| (\vec u(t), \p_t \vec u(t)) \|_X<\infty. 
\]

We put 
\[
\delta =\frac{\alpha}{2} 
\bigl\{ E (\vec \phi_{\omega}, \vec \psi_{\omega}) -E (\vec u_0, \vec v_0) \bigr\}. 
\]
Then, since $(\vec u_0, \vec v_0) \in \mathcal B_{\omega}$, 
we have $\delta>0$. 
Moreover, it follows from Proposition \ref{key}, Lemma \ref{InvSet2} 
and the conservation laws of $E$ and $Q$ that 
\[
H (\vec u(t), \p_t \vec u(t)) 
\le \alpha E (\vec u_0, \vec v_0) -\alpha E (\vec \phi_{\omega}, \vec \psi_{\omega})
=-2 \delta, 
\quad t\in [0, \infty). 
\]
Furthermore, 
since $\vec u(t)$ and $\p_t \vec u(t)$ are radially symmetric for all $t\in [0, \infty)$, 
it follows from Proposition \ref{virial2} that 
\begin{align*}
&-\frac{d}{dt} I_{\rho} (\vec u(t), \p_t \vec u(t)) \\
&\le H (\vec u(t), \p_t \vec u(t)) 
+C_1 \rho^{-1} \|\vec u(t)\|_{H^1}^2
+C_2 \rho^{-(N-1)/2} \|\vec u(t)\|_{H^1}^3 \\
&\le -2 \delta 
+C_1 C_0^2 \rho^{-1} +C_2 C_0^3 \rho^{-(N-1)/2} 
\end{align*}
for all $t\in [0, \infty)$ and $\rho>0$. 
Thus, there exists $\rho_0>0$ such that 
\[
\frac{d}{dt} I_{\rho_0} (\vec u(t), \p_t \vec u(t)) \ge \delta, 
\quad t\in [0, \infty), 
\]
which implies that 
\begin{equation} \label{Irho1}
\lim_{t \to \infty}  I_{\rho_0} (\vec u(t), \p_t \vec u(t)) 
=\infty. 
\end{equation} 

On the other hand, by the definition of $I_{\rho_0}$, 
there exists a constant $C_3=C_3(\rho_0, N)>0$ such that 
\[
I_{\rho_0} (\vec u(t), \p_t \vec u(t)) 
\le C_3 \|\vec u(t)\|_{H^1} \| \p_t \vec u(t) \|_{L^2}
\le C_3 C_0^2, 
\quad t\in [0, \infty), 
\] 
which contradicts \eqref{Irho1}. 

Hence, we conclude that 
$T_{\max} (\vec u_0, \vec v_0)<\infty$. 
\end{proof}

\begin{lemma} \label{blowup3} 
Assume \eqref{SC2}. Then, 
$\bigl( (\vec \phi_{\omega})^{\lambda},  \, (\vec \psi_{\omega})_{\lambda} \bigr) 
\in \mathcal B_{\omega}$ 
for all $\lambda \in (1, \infty)$. 
\end{lemma}

\begin{proof}
By \eqref{El4} and Lemma \ref{blowup2}, 
we have 
\[
\p_{\lambda} E \bigl( (\vec \phi_{\omega})^{\lambda}, 
\, (\vec \psi_{\omega})_{\lambda} \bigr) 
|_{\lambda=1} 
=H (\vec \phi_{\omega}, \vec \psi_{\omega}) =0, \quad  
\p_{\lambda}^2 E \bigl( (\vec \phi_{\omega})^{\lambda}, 
\, (\vec \psi_{\omega})_{\lambda} \bigr) 
|_{\lambda=1} \le 0.
\]
Thus, by \eqref{El1}, we see that 
$E \bigl( (\vec \phi_{\omega})^{\lambda}, \, (\vec \psi_{\omega})_{\lambda} \bigr) 
<E ( \vec \phi_{\omega}, \, \vec \psi_{\omega})$ 
for $\lambda \in (1, \infty)$. 

Moreover, by \eqref{def-H}, we have 
\[
H \bigl( (\vec \phi_{\omega})^{\lambda}, \, (\vec \psi_{\omega})_{\lambda} \bigr) 
=\lambda \p_{\lambda} 
E \bigl( (\vec \phi_{\omega})^{\lambda}, \, (\vec \psi_{\omega})_{\lambda} \bigr),
\]
and we see that 
$H \bigl( (\vec \phi_{\omega})^{\lambda}, \, (\vec \psi_{\omega})_{\lambda} \bigr)<0$ 
for $\lambda \in (1, \infty)$. 

\vspace{1mm}
Next, by \eqref{Qlambda}, we have 
$Q \bigl( (\vec \phi_{\omega})^{\lambda}, \, (\vec \psi_{\omega})_{\lambda} \bigr) 
=Q( \vec \phi_{\omega}, \, \vec \psi_{\omega})$ 
for $\lambda \in (0, \infty)$. 

\vspace{1mm}
Finally, since $P_{\omega} \bigl( \vec \phi_{\omega} \bigr) =0$ and 
\[
P_{\omega} \bigl( (\vec \phi_{\omega})^{\lambda} \bigr) 
=\alpha \lambda^{\alpha} M_{\omega} (\vec \phi_{\omega})
-(\alpha+2) \lambda^{\alpha+2} L(\vec \phi_{\omega}), 
\]
we see that 
$P_{\omega} \bigl( (\vec \phi_{\omega})^{\lambda} \bigr)<0$ 
for $\lambda \in (1, \infty)$. 

This completes the proof. 
\end{proof}

Finally, we give the proof of Theorem \ref{thm1}. 

\begin{proof}[Proof of Theorem \ref{thm1}]
Since $\vec \phi_{\omega}$ is radially symmetric and 
\[
\lim_{\lambda \to 1} 
\left\| \bigl( (\vec \phi_{\omega})^{\lambda}, \, (\vec \psi_{\omega})_{\lambda} \bigr)
-( \vec \phi_{\omega}, \, \vec \psi_{\omega} ) \right\|_X=0, 
\]
Theorem \ref{thm1} follows from 
Lemma \ref{blowup1} and Lemma \ref{blowup3}. 
\end{proof}

\section{Proof of Proposition \ref{key}} \label{sect5}

In this section, 
we give the proof of Proposition \ref{key}. 
We use a similar argument introduced by the author \cite{ohta18} 
for the nonlinear Schr\"odinger equation \eqref{nls-harmonic} 
with harmonic potential. 
First, we prove two simple lemmas. 

\begin{lemma} \label{VC4} 
If $(\vec u, \vec v) \in X$ satisfies 
$Q(\vec u, \vec v)=Q (\vec \phi_{\omega}, \vec \psi_{\omega})$ 
and 
$L (\vec u)=L (\vec \phi_{\omega})$, 
then 
$E (\vec \phi_{\omega}, \vec \psi_{\omega}) \le E (\vec u, \vec v)$. 
\end{lemma}

\begin{proof}
Since $L (\vec u)=L (\vec \phi_{\omega})$, 
it follows from Lemma \ref{VC3} that 
$P_{\omega} (\vec u)\ge 0$. 
Moreover, since 
$\alpha J_{\omega} (\vec u)=P_{\omega} (\vec u)+2 L (\vec u)$ 
and 
$P_{\omega} (\vec \phi_{\omega})=0$, 
we have 
\begin{align*}
\alpha J_{\omega}(\vec \phi_{\omega})
&=P_{\omega} (\vec \phi_{\omega})+2 L (\vec \phi_{\omega})
=2 L (\vec \phi_{\omega}) \\
&=2 L (\vec u) 
\le P_{\omega} (\vec u)+2 L (\vec u)
=\alpha J_{\omega} (\vec u). 
\end{align*}

Thus, by \eqref{Somega} and \eqref{Somega0}, we have 
\begin{align*}
E (\vec \phi_{\omega}, \vec \psi_{\omega}) 
&=S_{\omega} (\vec \phi_{\omega}, \vec \psi_{\omega}) 
+\omega Q(\vec \phi_{\omega}, \vec \psi_{\omega}) 
=J_{\omega} (\vec \phi_{\omega}) 
+\omega Q(\vec \phi_{\omega}, \vec \psi_{\omega}) \\
&\le J_{\omega} (\vec u) +\omega Q(\vec u, \vec v)
\le S_{\omega} (\vec u, \vec v) +\omega Q(\vec u, \vec v)
=E(\vec u, \vec v). 
\end{align*}

This completes the proof. 
\end{proof}

\begin{lemma} \label{function-g}
Let $\beta>1$, and define 
\begin{equation} \label{def-g}
g(s)=s^{\beta}-1-\beta (s-1)-\frac{\beta (\beta-1)}{2} s^{\beta-1} (s-1)^2
\end{equation} 
for $s\in (0, \infty)$. 
Then, $g(s)>0$ for all $s \in (0,1)$. 
\end{lemma}

\begin{proof}
By the Taylor expansion of $s^{\beta}$ at $s=1$, 
for every $s\in (0,1)$ there exists $\xi(s) \in (s, 1)$ such that 
\[
s^{\beta}=1+\beta (s-1)+\frac{\beta (\beta-1)}{2} \xi(s)^{\beta-2} (s-1)^2, 
\]
which implies that 
\[
g(s)=\frac{\beta (\beta-1)}{2} (s-1)^2
\bigl\{ \xi(s)^{\beta-2}-s^{\beta-1} \bigr\}. 
\]

Since $\beta>1$ and $s<\xi(s)<1$, 
we have 
$s^{\beta-1}<\xi(s)^{\beta-1}<\xi(s)^{\beta-2}$. 

Thus, we have $g(s)>0$ for $s\in (0,1)$. 
\end{proof}

Finally, we give the proof of Proposition \ref{key}. 

\begin{proof}[Proof of Proposition \ref{key}]
Let $(\vec u, \vec v) \in X$ satisfy \eqref{HQP}. 
Then, by Lemma \ref{VC3}, we have 
\begin{equation} \label{HQL}
H (\vec u, \vec v) \le 0, \quad 
Q (\vec u, \vec v) = Q (\vec \phi_{\omega}, \vec \psi_{\omega}), \quad 
L(\vec \phi_{\omega})<L(\vec u). 
\end{equation} 

We divide the proof into two cases: 
$K (\vec \psi_{\omega})\le K (\vec v)$ 
and 
$K (\vec v) \le K (\vec \psi_{\omega})$. 
In both cases, the function 
$f: (0, \infty) \to \R$ defined by 
\begin{align}
f (\lambda)
&= \alpha E (\vec u^{\lambda}, \vec v_{\lambda})
-\lambda^{\alpha} H (\vec u, \vec v) 
\label{def-f} \\
&=\alpha \left( \lambda^{-\alpha}+\lambda^{\alpha} \right) K(\vec v)
-\bigl\{ \alpha \lambda^{\alpha+2}-(\alpha+2) \lambda^{\alpha} \bigr\} L (\vec u)
\nonumber 
\end{align}
plays an important role. 
Our goal is to prove that 
$\alpha E (\vec \phi_{\omega}, \vec \psi_{\omega}) \le f(1)$. 

\vspace{1mm} \noindent 
(Case 1) \hspace{1mm} 
First, we assume that $K (\vec \psi_{\omega})\le K (\vec v)$. 

In this case, since 
$H (\vec \phi_{\omega}, \vec \psi_{\omega})=0$ by \eqref{El4} 
and $L(\vec \phi_{\omega})<L(\vec u)$, 
we have 
\begin{align*}
\alpha E (\vec \phi_{\omega}, \vec \psi_{\omega})
&=\alpha E (\vec \phi_{\omega}, \vec \psi_{\omega})
-H (\vec \phi_{\omega}, \vec \psi_{\omega}) 
=2 \alpha K (\vec \psi_{\omega})+2 L (\vec \phi_{\omega}) \\
&\le 2 \alpha K (\vec v)+2 L (\vec u) =f(1). 
\end{align*}

\noindent 
(Case 2) \hspace{1mm} 
Next, we assume that $K (\vec v) \le K (\vec \psi_{\omega})$. 

Let $\lambda_0$ be a positive constant satisfying 
$\lambda_0^{\alpha+2} L(\vec u)=L(\vec \phi_{\omega})$. 
Then, since $L(\vec \phi_{\omega})<L(\vec u)$, 
we have $0<\lambda_0<1$, and 
\begin{equation} \label{LQ1} 
L\bigl( \vec u^{\lambda_0} \bigr)=L(\vec \phi_{\omega}), \quad 
Q( \vec u^{\lambda_0}, \vec v_{\lambda_0} ) 
= Q (\vec u, \vec v)
= Q (\vec \phi_{\omega}, \vec \psi_{\omega}). 
\end{equation} 
Thus, by Lemma \ref{VC4}, we have 
$E (\vec \phi_{\omega}, \vec \psi_{\omega}) 
\le E (\vec u^{\lambda_0}, \vec v_{\lambda_0})$. 

Moreover, since $H(\vec u, \vec v)\le 0$,  we have 
\[
\alpha E (\vec \phi_{\omega}, \vec \psi_{\omega}) 
\le \alpha E (\vec u^{\lambda_0}, \vec v_{\lambda_0}) 
\le f(\lambda_0). 
\]

Therefore, it suffices to prove that $f(\lambda_0)\le f(1)$. 
Note that $f(\lambda_0)\le f(1)$ is equivalent to 
\begin{equation} \label{flam1}
\alpha \left( \lambda_0^{-\alpha}+\lambda_0^{\alpha} -2 \right) K(\vec v)
\le 
\bigl\{ \alpha \lambda_0^{\alpha+2}-(\alpha+2) \lambda_0^{\alpha} +2 \bigr\} L (\vec u). 
\end{equation} 

By the assumption of Case 2, Lemma \ref{blowup2} and \eqref{LQ1}, 
we have 
\begin{equation} \label{flam0}
\alpha^2 K (\vec v)
\le \alpha^2 K (\vec \psi_{\omega})
\le (\alpha+2) L ( \vec \phi_{\omega})
=(\alpha+2) \lambda_0^{\alpha+2} L (\vec u). 
\end{equation} 
Thus, we see that \eqref{flam1} holds if 
\begin{equation} \label{flam2}
(\alpha+2) \lambda_0^{\alpha+2} (\lambda_0^{-\alpha}+\lambda_0^{\alpha} -2)
\le \alpha \bigl\{ \alpha \lambda_0^{\alpha+2}-(\alpha+2) \lambda_0^{\alpha} +2 \bigr\}. 
\end{equation} 

Note that \eqref{flam2} is equivalent to 
\[
\frac{1}{\alpha} \left( 1+\frac{2} {\alpha} \right) 
\lambda_0^2 (\lambda_0^{\alpha} -1)^2 
\le 
\lambda_0^{\alpha+2}-1
-\left( 1+\frac{2} {\alpha} \right) (\lambda_0^{\alpha}-1). 
\]
Thus, we see that \eqref{flam2} is equivalent to $g(\lambda_0^{\alpha}) \ge 0$, 
where $g$ is the function defined by \eqref{def-g}  
with $\beta=1+2/\alpha$. 

Therefore, by Lemma \ref{function-g}, 
we have $g(\lambda_0^{\alpha})>0$ and $f(\lambda_0)\le f(1)$. 

\vspace{1mm} 
This completes the proof of Proposition \ref{key}. 
\end{proof} 

\bigskip \noindent 
\textbf{Acknowledgements.} \hspace{0mm} 
This research started when the author visited Hokkaido University 
and gave a one-week intensive lecture 
in June 2024.  
The author would like to express his gratitude to Professor Satoshi Masaki 
for his warm hospitality.
The author also thanks Professors Hayato Miyazaki and Kota Uriya 
for valuable discussions. 
This work was supported by 
JSPS KAKENHI Grant Number JP24K06803.

\end{document}